\documentclass{article}
\usepackage[english]{babel}
\usepackage{amsmath, amsxtra, latexsym,amscd, amsthm, amssymb,indentfirst}
\usepackage[mathscr]{eucal}
\usepackage{enumerate}
\usepackage{textcomp}
\usepackage{array, tabularx, longtable}%
\usepackage{multicol}
\usepackage{color}
\usepackage{algorithm}%
\usepackage{algorithmicx}%
\usepackage{algpseudocode}%
\usepackage{listings}%
\usepackage{enumitem}
\usepackage{graphicx}
\usepackage{subcaption}
\usepackage{amsthm}

\theoremstyle{plain}%
\newtheorem{theorem}{Theorem}[section]
\newtheorem{corollary}{Corollary}[section]%  meant for continuous numbers
\newtheorem{lemma}{Lemma}[section]% meant for sectionwise numbers
\newtheorem{remark}{Remark}%
\theoremstyle{thmstyletwo}%
\newtheorem{example}{Example}%

\theoremstyle{thmstylethree}%
\newtheorem{definition}{Definition}%

\title{Fixed-Time Convergence of Time-Varying
	Neurodynamic Systems for Mixed Variational
	Inequalities}

\author{Vajahat Karim Khan\thanks{Department of Mathematics, Aligarh Muslim University, Aligarh, U.P, India; email: gh7904@myamu.ac.in, gi9766@myamu.ac.in, ahmad\_kalimuddin@yahoo.co.in}, Vijendra Kumar Varshney and Md. Kalimuddin Ahmad
}

\begin{document}
	\maketitle

\abstract{This paper proposes novel fixed-time (FXT) convergent neurodynamic approaches for solving mixed variational inequality problems (MVIs). A class of first-order proximal neurodynamic models (PNMs), including time-varying proximal neurodynamic models (TVPNMs), is developed to guarantee FXT convergence to the solution of MVIs from arbitrary initial conditions. Rigorous convergence and stability analyses are established under the assumptions of strong pseudomonotonicity and Lipschitz continuity, using Lyapunov stability theory. The proposed methods exhibit FXT convergence from any initial point, with convergence speed significantly enhanced through the strategic design of time-varying coefficients. Explicit upper bounds on the settling time are derived for the time-varying neurodynamic models. In addition, the robustness of the proposed approaches against bounded noise disturbances is analyzed. The applicability of the proposed framework is further demonstrated for composite optimization problems and minimax optimization problems. Also, numerical examples are presented to demonstrate the effectiveness and convergence behavior of the proposed methods.}

{\bf keywords} {Mixed Variational Inequalities, Proximal Dynamic Methods,  Strong pseudo-monotonicity, Convergence, Global Stability Analysis.}


%\pacs[MSC Classification] {65L10, 26D10, 90C31, 49J40, 49M30..}

\section{Introduction}\label{p3sec1}
Let $\mathbb{R}^l$ denote the real $l$-dimensional Euclidean space equipped with the inner product 
$\langle \cdot , \cdot \rangle$ and the induced norm $\|\cdot\|$. Variational inequality problems (VIPs) provide a powerful and unified mathematical framework for modeling a wide range of problems arising in optimization, equilibrium theory, signal processing, and control systems \cite{Facchinei03,Kinderlehrer00, Scutari10,Konnov12, Neittaanmaki88}. Due to their generality, VIPs can characterize constrained optimization problems, saddle-point formulations, nonlinear complementarity problems, and Nash equilibrium seeking problems \cite{Scutari10, JuX25, Korpelevich, Cavazzuti02}. A significant extension of VIPs is the class of mixed variational inequality problems (MVIPs), which incorporate an additional nonsmooth convex function and therefore have extensive applications in game theory, image processing, and engineering systems \cite{Zheng24,JuX21Neuro}.

Over the past decades, numerous numerical algorithms have been developed to solve VIPs and MVIPs, including projection methods, extragradient algorithms, forward--backward splitting schemes, and proximal-based methods \cite{Parikh14,Censor1,Tseng, Thong18, JuX21exp, Pham17, Malitsky20, Vuong20}. These techniques are typically implemented in discrete time and are well-suited for digital computation. However, their sequential nature and iterative structure often limit their applicability in real-time and large-scale settings, where fast convergence and parallel processing capabilities are crucial.

 To overcome these limitations, neurodynamic optimization approaches based on continuous-time dynamical systems have been widely investigated. Such approaches leverage the intrinsic parallelism of neural dynamics and can be efficiently realized via analog circuits or hardware platforms. Early neurodynamic models for variational inequalities focused mainly on asymptotic or exponential convergence, which implies that the equilibrium solution is reached only as time approaches infinity. Although sufficient in theory, infinite-time convergence is undesirable in practical applications that require rapid, predictable convergence behavior.

Motivated by these concerns, finite-time convergence has been introduced into neurodynamic optimization frameworks, enabling solutions to be reached within a finite settling time  \cite{JuX21Novel, JuX21Neuro, JuX25}. Nevertheless, a major drawback of finite-time stable systems is that the convergence time depends explicitly on the initial conditions, making it difficult to estimate or guarantee performance in advance. Fixed-time stability theory was subsequently proposed to address this issue by ensuring convergence within a uniform upper bound that is independent of initial states \cite{Garg23, JuX23, JuX25,Zheng22, Liao23, Nguyen23}. In recent years, fixed-time neurodynamic approaches have been successfully applied to variational inequality problems, Nash equilibrium seeking, and signal recovery tasks \cite{Zheng22, Polyakov12, JuX21Novel}

Despite these advances, most existing fixed-time neurodynamic models rely on constant gain parameters, which may result in conservative settling-time estimates and limited flexibility in tuning the convergence behaviour. To further improve convergence performance, time-varying neurodynamic approaches have been proposed, offering additional degrees of freedom for system design and allowing faster convergence with less conservative time bounds \cite{JuX25, Zheng24}. Moreover, robustness against bounded disturbances and noise is a critical requirement for practical implementations, particularly in hardware-based realizations.

Recently, time-varying neurodynamic approaches have emerged as an effective means of accelerating convergence and reducing conservatism in fixed-time stability analysis. By allowing system coefficients to vary with time, these approaches introduce additional degrees of freedom for shaping system trajectories and enhancing transient performance. Meanwhile, robustness against bounded disturbances and external noise has become an essential requirement, particularly for real-world implementations subject to hardware limitations and environmental perturbations.

Inspired by these observations, this paper proposes novel fixed-time (FXT) neurodynamic optimization approaches with time-varying coefficients for solving MVIPs. The proposed methods guarantee fixed-time convergence to the MVIPs solution from arbitrary initial conditions, while significantly improving the convergence speed through a suitable design of time-varying parameters. By employing Lyapunov stability theory, explicit upper bounds on the settling time are derived, demonstrating that the convergence time is independent of both the initial state and system parameters. Furthermore, the robustness of the proposed neurodynamic system against bounded disturbances and noise is rigorously analyzed.\\
\noindent
The main contributions of this paper are summarized as follows:\\
\noindent
\begin{enumerate}
\item A novel first-order proximal neurodynamic model (PNM) is proposed for solving MVIPs. Fixed-time convergence is rigorously established under the assumptions of strong pseudomonotonicity and Lipschitz continuity.
\item A class of fixed-time convergent time-varying proximal neurodynamic models (TVPNMs) is further developed. Sufficient conditions ensuring fixed-time stability are derived, and detailed convergence analyses are provided for each model.
\item The proposed neurodynamic framework is extended to composite optimization problems (COPs) and minimax optimization problems (MOPs). Numerical simulations are presented to validate the effectiveness
of the proposed methods.
\end{enumerate}
\section{Preliminary Results with Notations}\label{p3sec2}
 Let $\Omega$ be a closed, convex subset of $\mathbb{R}^l$, which is nonempty. Suppose that $\Psi : \mathbb{R}^l \rightarrow \mathbb{R} \cup \{+\infty\}$ is a proper, l.s.c and convex real-valued function, and $\Upsilon : \mathbb{R}^l \rightarrow \mathbb{R}^l$ is a vector-valued function. Then, mixed variational inequalities (MVI) $(\Upsilon, \Omega)$ problem seeks a point $w^* \in \Omega$ that satisfies the condition 
\begin{equation}\label{MV}
    \langle \Upsilon(w^*), y - w^* \rangle + \Psi(y) - \Psi(w^*) \geq 0, \quad \forall y \in \mathbb{R}^l.
\end{equation}
Therefore, the solution set of \(\operatorname{MVIP}(\Upsilon, \Omega )\) is denoted by \(\operatorname{Sol}(\Upsilon, \Omega)\). For every $\mu$, we defined the proximal operator $\operatorname{prox}_{\mu \Psi}(v)$ as:
\begin{equation}\label{eqproximal}
    \operatorname{prox}_{\mu \Psi}(v) = \arg\min_{z \in \Omega} \left\{ \Psi(z) + \frac{1}{2\mu} \|v - z\|^2 \right\},\quad\forall v\in\Omega.
\end{equation}
The operator $\operatorname{prox}_{\mu \Psi}$ plays a crucial role in the dynamic system.\\

\noindent
In this section, we introduce the definitions, assumptions, and lemmas required throughout the paper.
A function $\Upsilon : \Omega\subseteq
\mathbb{R}^l \to \mathbb{R}$ is said to be:
\begin{enumerate}
    \item \textbf{strongly monotone} with modulus \(\theta > 0\) on $\Omega$, if
    \[
    \langle \Upsilon(w) - \Upsilon(y), w - y \rangle
\ge\theta \|w - y\|^2,
\quad \forall w, y \in \Omega.
    \]
\item \textbf{monotone} on $\Omega$, if
    \[
        \langle \Upsilon(w)-\Upsilon(y), w - y \rangle
\ge 0,
        \quad \forall w, y \in \Omega.
    \]
\end{enumerate}  
   \textbf{Assumption (A1).} $\Upsilon$ is \textbf{strongly pseudo-monotone} with modulus \(\zeta > 0\) on $\Omega$, if  
    \[ \langle \Upsilon(w), y - w \rangle+ \Psi(y) - \Psi(w) \ge 0,
        \quad \forall w, y \in \Omega,
    \]
        implies
        \[
\langle \Upsilon(y), y - w \rangle+ \Psi(y) - \Psi(w)
\ge \zeta \|w - y\|^2,
\quad \forall w, y \in \Omega .
    \]
    \textbf{Assumption (A$1^*$).} $\Upsilon$ is \textbf{pseudo-monotone} on $\Omega$, if
     \[ \langle \Upsilon(w), y - w \rangle+ \Psi(y) - \Psi(w) \ge 0,
        \quad \forall w, y \in \Omega,
    \]
        implies
        \[
\langle \Upsilon(y), y - w \rangle+ \Psi(y) - \Psi(w)
\ge 0,
\quad \forall w, y \in \Omega .
    \]
    \textbf{Assumption (A2).} $\Upsilon$ is a \textbf{Lipschitz continuous} on $\Omega$, if there exists a constant \(L > 0\) such that  
\[
\|\Upsilon(w) - \Upsilon(y)\|
\le L \|w - y\|, \qquad \forall w,y \in \Omega.
\]
\textbf{Assumption (A3).} \(\mu L^{2} < 2\zeta\).\\
\begin{remark} The implications \((1) \Rightarrow (2)\), \((1) \Rightarrow (\textbf{A} 1)\), \((\textbf{A}  1) \Rightarrow (\textbf{A}  1^*)\), and 
\((2) \Rightarrow (\textbf{A} 1^*)\) are evident. Assumption \textbf{A2}, guarantees the uniqueness of a solution to $\operatorname{MVIP}(\Upsilon,\Omega)$.\\
\end{remark}
\begin{lemma}\label{lemFXtime}\cite{TVl1}
Consider the following time-varying system:
\begin{equation}\label{ODM}
\dot{w}(t) = \mathcal{G}(t,w(t)), \qquad w(t_0)=w_0,
\end{equation}
where $w:[t_0,+\infty)\to\mathbb{R}^l$, and
$\mathcal{G}:[t_0,+\infty)\times\mathbb{R}^l\to\mathbb{R}^l$ is a continuous function.
Suppose there exists a continuously differentiable positive definite
function $\mathfrak{V}:[t_0,+\infty)\times\mathbb{R}^l\to\mathbb{R}_+$ such that
\begin{equation}
\dot {\mathfrak{V}}(t,w(t))
\le
- \Theta_1 g_1(t) \mathfrak{V}^{p_1}(t,w(t))
- \Theta_2 g_2(t) \mathfrak{V}^{p_2}(t,w(t)),
\end{equation}
where $\Theta_1>0$, $\Theta_2>0$, $0<p_1<1$, and $p_2>1$ are constants,
\[
\dot {\mathfrak{V}}(t,w(t))
:=
\frac{\partial \mathfrak{V}}{\partial t}(t,w)
+
\frac{\partial \mathfrak{V}}{\partial w
}(t,w)\,\dot w(t),
\]
and $g_1,g_2:[t_0,+\infty)\to\mathbb{R}_{++}$ are continuous functions.
Furthermore, assume that
\begin{equation}
\frac{1}{\Theta_2(p_2-1)}
<
\int_{t_0}^{+\infty} g_2(s)\,ds,
\end{equation}
and
\begin{equation}
N_1\circ N_2^{-1}
\!\left(\frac{1}{\Theta_2(p_2-1)}\right)
+
\frac{1}{\Theta_1(1-p_1)}
<
\int_{t_0}^{+\infty} g_1(s)\,ds,
\end{equation}
where
\(N_1(t)=\int_{t_0}^{t} g_1(s)\,ds, \ \text{and} \ 
N_2(t)=\int_{t_0}^{t} g_2(s)\,ds 
\) are the functions respectively. Then, there exists a time $\mathcal{T}>t_0$ such that for all $t\ge\mathcal{T}$ and any solution
$w(t)$ of system \eqref{ODM}, the function $\mathfrak{V}(t,w(t))$ satisfies
$
\mathfrak{V}(t,w(t))=0$.
Moreover, the settling time $T$ satisfies
\begin{equation}
\mathcal{T}
\le
N_1^{-1}
\!\left(
N_1\circ N_2^{-1}
\!\left(\frac{1}{\Theta_2(p_2-1)}\right)
+
\frac{1}{\Theta_1(1-p_1)}
\right).
\end{equation}\\
\end{lemma}
\begin{definition}\cite{5,6}\label{EPdef1}
From the system \eqref{ODM}, we obtained
\begin{itemize}
    \item[(a)] 
    A point \(w^* \in\mathbb{R}^l\) is an \emph{equilibrium point} of \eqref{ODM} if \(\mathcal{G}(w^*) = 0.\)
    \item[(b)] 
    An equilibrium point \(w^*\) of \eqref{ODM} is said to be Lyapunov stability if, 
    for any \(\varepsilon > 0\), there exists \(\delta > 0\) such that 
    \[\|w_0-w^*\|<\delta\implies\forall t\ge t_0,\ \|w(t)-w^*\|<\varepsilon .\]
    \item [(c)] An equilibrium point \(w^*\) of \eqref{ODM} is said to be FT convergence if  there exists a FT, $\mathcal{T}:\mathbb{R}^l\setminus\{0\}\to (0,\infty)\quad (\text{i.e.}\ \mathcal{T} < \infty)$ such that for all $w(0) \in N \setminus \{0\}$,
    \[ \lim_{t \to \mathcal{T}} w(t)=0 \quad \text{and} \quad \mathcal{T} = \mathcal{T}(w(0)) < \infty.
    \]
\end{itemize}
\end{definition}

\begin{remark}
A solution \(w(t)\) of system \eqref{ODM} exists on an interval containing \(t_{0}=0\) if \(\mathcal{G}\) is continuous on its domain. Additionally, the uniqueness of the solution to \eqref{ODM} follows from the local Lipschitz continuity of \(\mathcal{G}\). Any equilibrium point of \eqref{ODM} which is not at the origin can be translated to the origin by a suitable change of variables (see \cite{7}). In particular, let \(\bar{w}\) be a nonzero equilibrium point of \eqref{ODM} and define \(u:= w-\bar{w}\). Then, in terms of the new variable \(u\),
\[
\dot{u}=\dot{w}=\mathcal{J}, \qquad u(0)=u_{0}\in \mathbb{R}^{l},
\]
the origin \(u=0\) becomes an equilibrium point, where \(\mathcal{J}:=\mathcal{G}(u+\bar{w})\). Hence, Definition~\eqref{EPdef1} applies whether the system $\dot{w}=\mathcal{G}(w)$ has its equilibrium point at the origin or not.

\end{remark}

\begin{definition}\cite{18}
The equilibrium point of \eqref{ODM} is said to be FXT stable, if it is FT stable and there exists an upper bound $\mathcal{T}_{\max} < \infty$ such that
\[
    \sup_{w_0 \in \mathbb{R}^l} \mathcal{T}(w_0) \leq \mathcal{T}_{\max}.
\]
\end{definition}
\begin{lemma}[Pythagoras identity]
For any $a, b, c \in \mathbb{R}^l$, it holds that $2\langle a - b, c - b \rangle = \|a - b\|^2 + \|c - b\|^2 - \|a - c\|^2$.\\
\end{lemma}

\begin{lemma}\cite{10}\label{EPvklem6}
Let $\mu > 0$, $\Psi : \mathbb{R}^l \rightarrow (-\infty, +\infty]$ be a l.s.c. function, which is proper and convex. We denote by $\operatorname{prox}_{\mu \Psi}$ the proximal operator defined in \eqref{eqproximal}.
 Then the following assertions are true:
\begin{enumerate}
    \item For any $w\in \mathbb{R}^l$, we have
    \[
    \nu = \text{prox}_{\mu \Upsilon}(w) \Leftrightarrow \langle \nu -w, a - \nu \rangle \geq \mu \Psi(\nu) - \mu \Psi(a), \quad \forall a \in \mathbb{R}^l.
    \]
    \item $\|\text{prox}_{\mu \Upsilon}(w) - \text{prox}_{\mu \Upsilon}(y)\| \leq \|w - y\|$, $\forall w, y \in \mathbb{R}^l$.\\
\end{enumerate}
\end{lemma}

\begin{lemma}\label{lemE2.5}\cite{21,22}
Assume that the conditions \textbf{A1}, \textbf{A2}, and \textbf{A3} hold, and $w^*$ is a solution to MVIP  \eqref{MV}. Let $\mathcal{H}(w) := \text{prox}_{\mu\Psi}(w -\mu\Upsilon(w))$ and $\Xi := \frac{1}{\sqrt{1+2\mu \zeta - \mu^2 L^2}}$. Then for every $w \in\mathbb{R}^l$ the following statements hold:
\begin{enumerate}[label=\textnormal{(\roman*)}]
    \item $\|\mathcal{H}(w) - w^* \| \leq \Xi \|w - w^* \|$.
    \item $\langle w -\mathcal{H}(w), w^* -\mathcal{H}(w) \rangle \leq \|w -\mathcal{H}(w) \|^2$.
    \item $\langle w - w^*, w -\mathcal{H}(w) \rangle \geq (1- \Xi) \|w - w^* \|^2$.
    \item $\|w -\mathcal{H}(w) \| \geq (1- \Xi) \|w - w^* \|$.
\end{enumerate}
\end{lemma}

\section{Main Results of FXT Convergence}
We establish a clear connection between the solutions of  \eqref{MV} and the equilibrium points of our PNMs.
\begin{theorem}\label{ThE1}
For any \(\mu> 0\), a point \(w \in\operatorname{Sol}(\Upsilon, \Omega)\)
iff  
$w=\mathcal{H}(w)$
\end{theorem}
\begin{proof}
From Lemma~\ref{EPvklem6} (1), for all $w\in \mathbb{R}^l$, we have
\begin{equation*}
    \begin{split}
        w&= \mathcal{H}(w)\\
&\iff
\langle (w - \eta \Upsilon(w)) -w, y - w \rangle+ \eta \Psi(w) \le \eta \Psi(y),
\quad \forall\, y \in \mathbb{R}^l.
    \end{split}
\end{equation*}
This inequality is equivalent to
\[
\eta \langle \Upsilon(w), y - w \rangle+ \eta \Psi(y) - \eta \Psi(w) \ge 0,
\quad \forall\, y \in \mathbb{R}^l.
\]
Dividing both sides by $\eta > 0$, we obtain
\begin{equation*}
    \begin{split}
        \langle \Upsilon(w), y - w \rangle+ \Psi(y) - \Psi(w) &\ge 0,
\quad \forall w, y \in \mathbb{R}^l.
\end{split}
\end{equation*}
which completes the proof.
\end{proof}
\subsection{FXT Converging TVPNMs}
We propose the following FXT converging TVPNM for solving $\operatorname{MVIP}(\Upsilon,\Omega)$:
\begin{equation}\label{OD3}
\dot{w}(t) = - e(w(t))\,\phi(w(t)),
\end{equation}
where
\[
\phi(w) := w -\mathcal{H}(w), \qquad
\mathcal{H}(w) := \operatorname{prox}_{\mu \Psi}(w -\mu \Upsilon(w)),
\]
and
\[
e(w):=
\begin{cases}
\dfrac{\Upsilon_1(t)}{\|\phi(w)\|^{1-\rho_1}}
+\dfrac{\Upsilon_2(t)}{\|\phi(w)\|^{1-\rho_2}}
+\dfrac{\Upsilon_3(t)}{\|\phi(w)\|},
& w\notin \mathrm{Fix}(\mathcal{H}),\\[2mm]
0, & w \in \mathrm{Fix}(\mathcal{H}).
\end{cases}
\]
Here $\Upsilon_i(t):[t_0,+\infty)\to\mathbb{R}_{++}$ $(i=1,2,3)$ are continuous
functions, $0<\rho_1<1$ and $\rho_2>1$.
The TVPNM~\eqref{OD3} allows greater flexibility through its time-varying
coefficients. Obviously, $w^*$ is an equilibrium point of TVPNM~\eqref{OD3}
iff $w^*\in \mathrm{Sol}(\Upsilon, \Omega)$. 
We can easily verify that $w^*\in\text{Fix}(\mathcal{H}) = \operatorname{Sol}(\Upsilon, \Omega)= \{w^* \in \mathbb{R}^l : \phi(w^*) = 0\}$. Moreover, note that TVPNM \eqref{OD3} can be written in an equivalent form as follows:
\begin{equation}
\dot{w} = -\Upsilon_1(t) \frac{\kappa(w)}{\mathcal{T}_p} \left[ \|\phi(w)\|^{\rho_1} \left( 1 + \frac{\Upsilon_2(t)}{\Upsilon_1(t)} \|\phi(w)\|^{\rho_2 - \rho_1} \right) + \frac{\Upsilon_3(t)}{\Upsilon_1(t)} \right], \tag{13}
\end{equation}
where 
\[
\kappa(w) := \frac{\phi(w)}{\|\phi(w)\|} = \frac{w -\mathcal{H}(w)}{\|w -\mathcal{H}(w)\|} = \frac{w -\mathcal{H}(w)}{\|w -\mathcal{H}(w)\|}.
\]

\begin{remark}
\begin{enumerate}
    \item In the TVPNM~\eqref{OD3}, the term $-\kappa(w)$ can be interpreted as the search direction, while $\lambda_t := \Upsilon_1(t)\left[\|\phi(w)\|^{\rho_1}\left(1 + \frac{\Upsilon_2(t)}{\Upsilon_1(t)}\|\phi(w)\|^{\rho_2-\rho_1}\right) + \frac{\Upsilon_3(t)}{\Upsilon_1(t)}\right]$
represents the stepsize at the current iterate $w= w(t)$. Once the functions $\Upsilon_1(t),\Upsilon_2(t),\Upsilon_3(t)$ and $\rho_1,\rho_2$ are specified, the stepsize $\lambda_t$ is automatically updated with respect to $t$. Therefore, TVPNM~\eqref{OD3} may also be viewed as a TVPNM with a self-adaptive dynamical step size.\\
\item By taking  $\Upsilon_3(t) = 1$, TVPNM \eqref{OD3} is reduce to
\[
\dot{w} = -\Upsilon_1(t) \kappa(w) \left[ \|\phi(w)\|^{\rho_1} \left( 1 + \frac{\Upsilon_2(t)}{\Upsilon_1(t)} \|\phi(w)\|^{\rho_2 - \rho_1} \right) + \frac{1}{\Upsilon_1(t)} \right].
\]
\item By taking $\Upsilon_1(t) = \Upsilon_2(t) = 0$ in \eqref{OD3}, TVPNM \eqref{OD3} is reduce to the normalized TVPNM (NL-TVPNM):
\begin{equation}
\dot{w} = -\Upsilon_3(t)\kappa(w).
\end{equation}
\item By taking $\Upsilon_2(t) = \Upsilon_3(t) = 0$ in \eqref{OD3}, TVPNM \eqref{OD3} reduces to the fixed-time convergent TVPNM (FXT-TVPNM) given below:
\begin{equation}\label{ODE4}
\dot{w} = -\Upsilon_1(t) \kappa(w) \|w -\mathcal{H}(w)\|^{\rho_1}. 
\end{equation}
\item When $\Upsilon_3(t) = 0$, the TVPNM~\eqref{OD3} reduces to TVPNM in~\cite{20}. For $\Upsilon_3(t) > 0$, the stepsize $\lambda_t$ in TVPNM~\eqref{OD3} is larger than in the case $\Upsilon_3(t) = 0$.\\
\end{enumerate}
\end{remark}

\begin{lemma}\label{Cor5.1}
Let \(w^*\) be a solution of \(\operatorname{MVIP}(\Upsilon, \Omega )\) \eqref{MV} iff it is an equilibrium point of  TVPNM~\eqref{OD3}.
\end{lemma}

\begin{proof}
From TVPNM \eqref{OD3}, we let $w^*\in \mathbb{R}^{l}$  equilibrium point. Then, we obtain 
\begin{align*}
   \dot{w}(t) = 0 
&\Rightarrow 
-e(w^*)\bigl(w^* -\mathcal{H}(w^*)\bigr) = 0, \\
&\Rightarrow \quad w^* =\mathcal{H}(w^*),
\end{align*}
Therefore, we obtain that \(w^* \in \operatorname{Fix}(\mathcal{H})\). As Theorem~\eqref{ThE1}, \(w^*\) is solution of \(\operatorname{MVIP}(\Upsilon, \Omega )\) \eqref{MV}, and the reverse implication is also true.
\end{proof}

\subsection{Convergence Behavior Analysis of the TVPNM}
Next, we study how TVPNM \eqref{OD3} achieves FXT convergence under the given assumptions.\\
\begin{theorem}\label{thmFT}
Assume that $\Upsilon$ satisfies assumptions \textbf{A1--A3}.
Suppose further that
\[
\frac{1}{\Theta_2(p_2-1)} < \int_{t_0}^{+\infty} \Upsilon_2(s)\,ds,\ \text{and} \quad \mathcal{F}_1\circ \mathcal{F}_2^{-1}
\!\left(\frac{1}{\Theta_2(p_2-1)}\right)
+\frac{1}{\Theta_1(1-p_1)}
<
\int_{t_0}^{+\infty} \Upsilon_1(s)\,ds,\]
where
\[
p_1=\frac{1+\rho_1}{2}\in(0,1),\quad
p_2=\frac{1+\rho_2}{2}>1,\Theta_1 =2^{p_1}(1- \Xi),\quad
\Theta_2 =2^{p_2}(1- \Xi)^{\rho_2},
\]
and
\[\Xi=\frac{1}{\sqrt{1+2\mu\zeta-\mu^2L^2}},\quad\mathcal{F}_i(t)=\int_{t_0}^t \Upsilon_i(s)\,ds,\quad i=1,2.
\]
Then the trajectory $w(t)$ of system \eqref{OD3} converges to the unique solution
$w^*$ of $\operatorname{MVIP}(\Upsilon,\Omega)$ in fixed time, and the settling time
satisfies
\begin{equation}\label{eqtime}
\mathcal{T} \le
\mathcal{F}_1^{-1}
\!\left(
\mathcal{F}_1\circ \mathcal{F}_2^{-1}
\!\left(\frac{1}{\Theta_2(p_2-1)}\right)
+\frac{1}{\Theta_1(1-p_1)}
\right).
\end{equation}
\end{theorem}

\begin{proof}
Let $w^*$ be the solution of $\operatorname{MVIP}(\Upsilon, \Omega)$ in \eqref{MV}. By Lemma \eqref{Cor5.1}, $w^*$ serves as the unique equilibrium point of the TVPNM \eqref{OD3}. Consequently, $\operatorname{Fix}(\mathcal{H})=\{w^*\}$ and $\phi(w^*)=0$. Furthermore, using Lemma \eqref{lemE2.5}(iii), we obtain
\begin{equation}\label{eqEOD31}
    \langle w - w^*, \kappa(w) \rangle \geq \frac{(1- \Xi)\|w - w^*\|^2}{\|\phi(w)\|}, \quad \forall w \in \mathbb{R}^l,
\end{equation}
and
\[
\langle w - w^*, \phi(w) \rangle \geq (1- \Xi)\|w - w^*\|^2 > 0, \quad \forall w \in \mathbb{R}^l \setminus \{w^*\}.
\]
Thus, using Proposition 2 in \cite{26}, we conclude that for every initial condition there exists a unique solution to the TVPNM \eqref{OD3}. Next, we introduce the Lyapunov function candidate:
\begin{equation}\label{ODElyp}
    \mathfrak{V}(w) = \frac{1}{2} \|w - w^*\|^2.
\end{equation}
Differentiating \eqref{ODElyp}, we obtain the following
\begin{equation}
    \begin{split}
        \dot{\mathfrak{V}} &= (w - w^*)^{\top} \dot{w} \\
&= -\left(\Upsilon_1(t) \|\phi(w)\|^{\rho_1} + \Upsilon_2(t) \|\phi(w)\|^{\rho_2} + \Upsilon_3(t) \right) \langle w - w^*, \kappa(w) \rangle \\
&\leq -\Upsilon_1(t)(1- \Xi) \frac{\|w - w^*\|^2}{\|\phi(w)\|^{1-\rho_1}} - \Upsilon_2(t)(1- \Xi) \frac{\|w - w^*\|^2}{\|\phi(w)\|^{1-\rho_2}} - (1- \Xi)\Upsilon_3(t)\frac{\|w - w^*\|^2}{\|\phi(w)\|} \\
&\leq -\frac{\Upsilon_1(t)(1- \Xi)}{\Lambda^{1-\rho_1}} \|w - w^*\|^{1+\rho_1} - (1- \Xi)^{\rho_2} \left( \Upsilon_2(t) + \Upsilon_3(t) \|\phi(w)\|^{\rho_2} \right) \|w - w^*\|^{1+\rho_2}.
    \end{split}
\end{equation}
From \eqref{eqEOD31}, the first inequality follows, whereas using Lemma~\eqref{lemE2.5}(iv) together with Theorem 4.2 in~\cite{25}, we deduce the second inequality. Furthermore, since $2\zeta > \mu L^{2}$, it follows that, $1 < \Lambda = \dfrac{4\zeta}{4\zeta - \mu L^{2}} < 2$. Then
\begin{equation}\label{eqvk1}
    \dot{\mathfrak{V}} \leq -Q_1(\rho_1,t)(1- \Xi) \|w - w^*\|^{1+\rho_1} - Q_2(\rho_2, t)(1- \Xi)^{\rho_2} \|w - w^*\|^{1+\rho_2},
\end{equation}
where $Q_1(\rho_1,t) := \frac{\Upsilon_1(t)}{\Lambda^{1-\rho_1}}$ and $Q_2(\rho_2, t):= (\Upsilon_2(t) + \Upsilon_3(t) \|\phi(w)\|^{\rho_2})$. Then, it follows from \eqref{ODElyp} and \eqref{eqvk1} that
\begin{align}
\notag\dot{\mathfrak{V}} &\leq -Q_1(\rho_1,t) (1- \Xi)\left( \|w - w^*\|^2 \right)^{\frac{1+\rho_1}{2}} - Q_2(\rho_2, t)(1- \Xi)^{\rho_2}\left( \|w - w^*\|^2 \right)^{\frac{1+\rho_2}{2}} \\
&\leq -\left( \Theta_1 Q_1(\rho_1,t)\mathfrak{V}(t,w)^{p_1} + \Theta_2 Q_2(\rho_2,t)  \mathfrak{V}(t,w)^{p_2} \right),
\end{align}
where $\Theta_1 := 2^{p_1}(1-\Xi)>0$ with $p_1 := \frac{1+\rho_1}{2}\in(0.5,1)$ for $\rho_1\in(0,1)$, and 
$\Theta_2 := 2^{p_2}(1-\Xi)^{\rho_2}$ with $p_2 := \frac{1+\rho_2}{2}\in(1,+\infty)$ for $\rho_2\in(1,+\infty)$.
Moreover, $\mathcal{F}_i(t)=\int_{t_0}^t \Upsilon_i(s)\,ds$, $i=1,2$, and the settling time is given by \eqref{eqtime}. Thus, from Lemma~\eqref{lemFXtime}, together with Lemma~\eqref{Cor5.1}, it follows that the equilibrium point $w^*$ of system~\eqref{OD3} achieves FXT convergence.  Hence, this completes the proof.
\end{proof}

\subsection{Special Case}
In this subsection, we study a special VI that can be obtained from the MVI \eqref{MV} when the function $\Psi$ is chosen as the indicator function $\delta$ of a nonempty, closed, and convex set $\Omega \subseteq \mathbb{R}^l$. In this case, the VI is formulated as finding $w\in\Omega$ such that
\[
\langle \Upsilon(w), y-w\rangle \ge 0, \quad \forall y\in\Omega.
\]
In this setting, the proximal operator reduces to
\(
\operatorname{prox}_{\mu\Psi}(w)
=\operatorname{prox}_{\mu\delta}(w)
= P_{\Omega}(w),
\)
where
\(
P_{\Omega}(w):=\arg\min_{y\in\Omega}\|w-y\|
\)
denotes the metric projection of $w$ onto the set $\Omega$. If one selects
\(
\Upsilon_1(t)=\beta_1>0,
\ 
\Upsilon_2(t)=\beta_2>0,\ \text{and} \
\Upsilon_3(t)=\beta_3>0
\) in the system \eqref{OD3}.
Then, TVPNM \eqref{OD3} can be rewritten as the following FXT convergent projection dynamical model (FXT-PDM):
\begin{equation}\label{seqODE2}
    \dot{w}
=
-\beta_1 \frac{\phi(w)}{\|\phi(w)\|^{1-\rho_1}}
-\beta_2 \frac{\phi(w)}{\|\phi(w)\|^{1-\rho_2}}
-\beta_3 \frac{\phi(w)}{\|\phi(w)\|},
\end{equation}
where $\phi(w):=w-P_{\Omega}(w-\mu\Upsilon(w))$ with parameters
$\beta_1,\beta_2,\beta_3>0$, $\rho_1\in(0,1)$, and $\rho_2>1$ are tunable, and the
right-hand side is set to zero whenever $w=P_{\Omega}(w-\mu\Upsilon(w))$.\\
\begin{corollary}\label{croFT}
Let the assumptions of Theorem~\eqref{thmFT} hold,
then the trajectory $w(t)$ of system~\eqref{seqODE2} converges to $w^*$ in FXT, and settling time satisfies:
\begin{equation}
  \mathcal{T}(w(0)) \leq \frac{1}{N_1(1 - p_1)} + \frac{1}{N_2(p_2 - 1)},  
\end{equation}
where $w(0)$ is an initial condition of FXT-PDM \eqref{seqODE2}, and \(p_1 := \frac{1+\rho_1}{2}, \quad p_2 := \frac{1+\rho_2}{2},\quad \Xi=\frac{1}{\sqrt{1+2\mu\zeta-\mu^2L^2}}\) with
\begin{equation}\label{OD3eq3}
    N_1 := 2^{\frac{1+\rho_1}{2}} \beta_1(1-\Xi) \left( \frac{4\zeta - \mu L^2}{4\zeta} \right)^{1-\rho_1}
\end{equation}
and
\begin{equation}\label{OD3eq4}
   N_2 \in \left[ 2^{p_2}(1- \Xi)^{\rho_2} \beta_2, \; 2^{p_2}(1- \Xi)^{\rho_2} \left( \beta_2 + \beta_3 \|\phi(w(0))\|^{\rho_2} \right) \right]. 
\end{equation}
\end{corollary}
\subsection{Robustness Analysis}
In this subsection, we investigate the robustness of the proposed FXT
converging time-varying proximal neurodynamic model (TVPNM) when it is subject
to an external disturbance term
$\mathcal{D}:\mathbb{R}^l \to \mathbb{R}^l$.
Specifically, we consider the following disturbed system: (PNM-D)
\begin{equation}\label{OD3-D}
\dot{w}(t)
=
- e(w(t))\,\phi(w(t))
+ \mathcal{D}(w(t)),
\end{equation}
where
\[
\phi(w) := w -\mathcal{H}(w), \qquad
\mathcal{H}(w):=\operatorname{prox}_{\mu\Upsilon}(w -\mu\Upsilon(w)),
\]
and the gain function $e(w)$ is defined as:
\[
e(w):=
\begin{cases}
\dfrac{\Upsilon_1(t)}{\|\phi(w)\|^{1-\rho_1}}
+\dfrac{\Upsilon_2(t)}{\|\phi(w)\|^{1-\rho_2}}
+\dfrac{\Upsilon_3(t)}{\|\phi(w)\|},
& w\notin \mathrm{Fix}(\mathcal{H}),\\[2mm]
0, & w \in \mathrm{Fix}(\mathcal{H}).
\end{cases}
\]
\medskip
\noindent
\textbf{Assumption $A^*$.}
There exists a constant $q>0$ such that the disturbance term
$\mathcal{D}(w)$ satisfies
\[
\|\mathcal{D}(w)\| \le q \|w-w^*\|,
\qquad \forall w\in\mathbb{R}^l,
\]
where $w^*\in\operatorname{Sol}(\Upsilon,\Omega)$ denotes the unique solution of MVI.

\medskip
It follows immediately from assumption~\textbf{$A^*$} that $w^*$ is an equilibrium point of the disturbed system~\eqref{OD3-D}.
The following result demonstrates that the proposed FXT converging system~\eqref{OD3-D} is robust against such external disturbances.\\

\begin{theorem}\label{thmFT-robust}
Assume that $\Upsilon$ satisfies assumptions A1--A3.
Suppose that
\(
\min\!\left\{
\frac{\Upsilon_1(t)}{\Lambda^{1-\rho_1}},
\,
\Upsilon_2(t) + \Upsilon_3(t) \|\phi(w)\|^{\rho_2}
\right\}>q,
\)
and
\(\frac{1}{\Theta_2(p_2-1)}
<
\int_{t_0}^{+\infty}
\Upsilon_2(s)\,ds,\)
\[
\mathcal F_1\!\circ\!\mathcal F_2^{-1}
\!\left(\frac{1}{\Theta_2(p_2-1)}\right)
+\frac{1}{\Theta_1(1-p_1)}
<
\int_{t_0}^{+\infty}
\Upsilon_1(s)\,ds
\]
hold for $\Theta_1 := 2^{p_1}>0$ and $\Theta_2 := 2^{p_2}(1-\Xi)^{\rho_2}$, with 
$p_1 := \frac{1+\rho_1}{2}\in(0.5,1)$ for $\rho_1\in(0,1)$ and 
$p_2 := \frac{1+\rho_2}{2}\in(1,+\infty)$ for $\rho_2\in(1,+\infty)$.
Furthermore,
\[
\mathcal{F}_1(t)=\int_{t_0}^t\Big((1-\Xi)\frac{\Upsilon_1(s)}{\Lambda^{1-\rho_1}}-q\Big)\,ds,\quad
\mathcal{F}_2(t)=\int_{t_0}^t\Big(\Upsilon_2(s)+\Upsilon_3(s)\|\phi(w)\|^{\rho_2}-q\Big)\,ds.
\]
Then, under \textbf{A*}, the trajectory $w(t)$ of system~\eqref{OD3-D}
converges to $w^*\in\operatorname{Sol}(\Upsilon,\Omega)$ in fixed time.
\end{theorem}
\begin{proof}
Let $w^*$ be the unique solution of $\operatorname{MVIP}(\Upsilon,\Omega)$. Consider the Lyapunov function
\[
\mathfrak{V}(w)=\frac12\|w-w^*\|^2.
\]
Differentiating $\mathfrak V$ along the trajectories of~\eqref{OD3-D}, we obtain
\[
\dot{\mathfrak V}
=
\langle w-w^*,\dot x\rangle
=
- e(w)\langle w-w^*,\phi(w)\rangle
+\langle w-w^*,\mathcal D(w)\rangle.
\]
From Lemma~\ref{lemE2.5}(iii), one has
\[
\langle w-w^*,\phi(w)\rangle
\ge (1-\Xi)\|w-w^*\|^2,
\quad \forall w\in\mathbb{R}^l,
\]
and from assumption~A4,
\[
\langle w-w^*,\mathcal D(w)\rangle
\le q\|w-w^*\|^2.
\]
Hence,
\begin{equation}\label{R1}
   \dot{\mathfrak V}
\le
-(1-\Xi)e(w)\|w-w^*\|^2
+q\|w-w^*\|^2. 
\end{equation}
Substituting the expression of $e(w)$ into~\eqref{R1} yields
\begin{align}\label{R2}
\notag\dot{\mathfrak V}\le&
-(1-\Xi)\Upsilon_1(t)\frac{\|w-w^*\|^2}{\|\phi(w)\|^{1-\rho_1}}
-(1-\Xi)\Upsilon_2(t)\frac{\|w-w^*\|^2}{\|\phi(w)\|^{1-\rho_2}}\\
&-(1-\Xi)\Upsilon_3(t)\frac{\|w-w^*\|^2}{\|\phi(w)\|}
+q\|w-w^*\|^2.
\end{align}
From \eqref{eqEOD31}, the first inequality follows, whereas using Lemma~\eqref{lemE2.5}(iv) together with Theorem 4.2 in~\cite{25}, we deduce the second inequality. Furthermore, since $2\zeta > \mu L^{2}$, it follows that $1 < \Lambda = \dfrac{4\zeta}{4\zeta - \mu L^{2}} < 2$. Then, we have
\[
\frac{\|w-w^*\|^2}{\|\phi(w)\|^{1-\rho_1}}
\le
\frac{1}{\Lambda^{1-\rho_1}}\|w-w^*\|^{1+\rho_1},
\]
and
\[
\frac{\|w-w^*\|^2}{\|\phi(w)\|^{1-\rho_2}}
\le
(1-\Xi)^{\rho_2-1}\|w-w^*\|^{1+\rho_2}.
\]
Further, define the set
\(
\Sigma:=\{w \in\mathbb{R}^l:\|w-w^*\|\le1\}.
\) Then we take case 1: $w \in\Sigma$.
Since $0<\rho_1<1$, it holds that
\(
\|w-w^*\|^2\le \|w-w^*\|^{1+\rho_1}.
\)
Using this inequality and the above estimates in~\eqref{R2}, we obtain
\begin{equation}\label{eqvk2}
    \begin{split}
   \dot{\mathfrak V}
&\le
-\Big((1-\Xi)\frac{\Upsilon_1(t)}{\Lambda^{1-\rho_1}}-q\Big)
\|w-w^*\|^{1+\rho_1}
\\&- (1- \Xi)^{\rho_2} \left( \Upsilon_2(t) + \Upsilon_3(t) \|\phi(w)\|^{\rho_2}\right) \|w - w^*\|^{1+\rho_2}+q\|w - w^*\|^{2}.     
    \end{split}
\end{equation}
Also, we take case 2: $x\notin\Sigma$. Since $\rho_2>1$ and $\|w-w^*\|>1$, one has
\(
\|w-w^*\|^2\le \|w-w^*\|^{1+\rho_2}.
\)
Thus, using this in ~\eqref{eqvk2}, we deduce

\begin{equation}\label{EOD3eq2}
\begin{split}
 \dot{\mathfrak V}
&\le
-\big((1-\Xi)\frac{\Upsilon_1(t)}{\Lambda^{1-\rho_1}}-q\big)
\|w-w^*\|^{1+\rho_1}\\&\qquad
-(1-\Xi)^{\rho_2}\Big(\left( \Upsilon_2(t) + \Upsilon_3(t) \|\phi(w)\|^{\rho_2} \right)-q\Big)
\|w-w^*\|^{1+\rho_2}\\
&\leq -Q_1(\rho_1,t)\|w - w^*\|^{1+\rho_1} - Q_2(\rho_2, t)(1- \Xi)^{\rho_2} \|w - w^*\|^{1+\rho_2},
    \end{split}
\end{equation}
where $Q_1(\rho_1,t) := (1- \Xi)\frac{\Upsilon_1(t)}{\Lambda^{1-\rho_1}}-q$ and $Q_2(\rho_2, t):= \left( \Upsilon_2(t) + \Upsilon_3(t) \|\phi(w)\|^{\rho_2} \right)-q
$. Then, it follows from \eqref{ODElyp} and \eqref{EOD3eq2} that
\begin{align}
\notag\dot{\mathfrak{V}} &\le -Q_1(\rho_1,t) \left( \|w - w^*\|^2 \right)^{\frac{1+\rho_1}{2}} - Q_2(\rho_2, t)(1- \Xi)^{\rho_2}\left( \|w - w^*\|^2 \right)^{\frac{1+\rho_2}{2}} \\
&\leq -\left( \Theta_1 Q_1(\rho_1,t)\mathfrak{V}(t,w)^{p_1} + \Theta_2 Q_2(\rho_2,t)  \mathfrak{V}(t,w)^{p_2} \right),
\end{align}
where $\Theta_1 := 2^{p_1}>0$ and $\Theta_2 := 2^{p_2}(1-\Xi)^{\rho_2}$, with 
$p_1 := \frac{1+\rho_1}{2}\in(0.5,1)$ for $\rho_1\in(0,1)$ and 
$p_2 := \frac{1+\rho_2}{2}\in(1,+\infty)$ for $\rho_2\in(1,+\infty)$. Thus, according to Lemma~\eqref{lemFXtime}, it follows that the equilibrium point $w^*$ of system~\eqref{OD3-D} achieves fixed-time convergence. 
\end{proof}

\section{Applications}
This section applies the proposed TVPNM \eqref{OD3} and FXT-TVPNM \eqref{ODE4} to COPs and MOPs.
\subsection{Composite Optimization Problems}
Here, we take COP:
\begin{equation}\label{eqCOP}
\min_{w\in \mathbb{R}^l} h(w) + \Psi(w),
\end{equation}
The function $\Psi: \mathbb{R}^l \to \mathbb{R}$ is assumed to be proper, lower semicontinuous, and convex, although it may fail to be differentiable. On the other hand, the function $h: \mathbb{R}^l \to \mathbb{R}$ is differentiable and has a Lipschitz continuous gradient with Lipschitz constant $L$.\\

In \cite{28}, the authors investigated the proximal gradient dynamical model (PGDM) corresponding to the COP \eqref{eqCOP}:
\begin{equation}\label{ODpgdm}
\dot{w} = -\rho(w - \text{prox}_{\mu \Upsilon}(w - \mu \nabla h(w))),
\end{equation}
where parameter $\rho > 0$ represents a scalar tuning gain. While the PGDM \eqref{ODpgdm} is known to exhibit exponential convergence, predefined-time convergence has not been considered. Motivated by this, we introduce the time-varying PGDM (TV-PGDM):
\begin{equation}\label{ODfpgdm}
\dot{w} = -e(w)(w - \text{prox}_{\mu \Upsilon}(w - \mu \nabla h(w))),
\end{equation}
where
\[
e(w) = 
\begin{cases}
\Upsilon_1(t) \frac{1}{\|\phi(w)\|^{1-\rho_1}} + \Upsilon_2(t) \frac{1}{\|\phi(w)\|^{1-\rho_2}} + \Upsilon_3(t) \frac{1}{\|\phi(w)\|}, & \text{if } w\in \mathbb{R}^l \backslash \text{Fix}(\mathcal{H}), \\
0, & \text{otherwise},
\end{cases}
\]
$\phi(w):=w -\operatorname{prox}_{\mu\Upsilon}(w -\mu\nabla h(w))$ and 
$\operatorname{Fix}(\mathcal{H}):=\{w:\phi(w)=0\}$. 
Assume that $\Upsilon_1(t)$, $\Upsilon_2(t)$, and $\Upsilon_3(t)$ are continuous functions, with tunable parameters $\rho_1\in(0,1)$ and $\rho_2>1$.

The TV-PGDM in \eqref{ODfpgdm} can be viewed as a generalization of the PGDM proposed in \cite{28}. In particular, by setting $\Upsilon_2(t) = \Upsilon_3(t) = 0$ in \eqref{ODfpgdm}, the TV-PGDM reduces to the FT PGDM.
\begin{lemma}
A point $w^*$ is an optimal solution of COP \eqref{eqCOP} iff it is an equilibrium point of TV-PGDM \eqref{ODfpgdm}
\end{lemma}

\begin{proof}
Let $\kappa(w) = \frac{w - \text{prox}_{\mu \Psi}(w - \mu \nabla h(w)))}{\|w - \text{prox}_{\mu \Psi}(w - \mu \nabla h(w))\|}$. If $w^*$ is an equilibrium point of TV-PGDM \eqref{ODfpgdm}, we get
\[
- e(w^*) (w^* - \text{prox}_{\mu \Psi}(w^* - \mu \nabla h(w^*))) = 0
\]
\[
\Leftrightarrow \kappa(w^*) (\Upsilon_1(t) \|\phi(w^*)\|^{\rho_1} + \Upsilon_2(t) \|\phi(w^*)\|^{\rho_2} + \Upsilon_3(t)) = 0
\]
\[
\Leftrightarrow \kappa(w^*) = 0\]
\[\Leftrightarrow \phi(w^*) = 0.
\]
This corresponds to the following expression:
\begin{equation}\label{eqcop1}
w^* = \text{prox}_{\mu \Psi}(w^* - \mu \nabla h(w^*)).
\end{equation}
By exploiting the characterization of the proximal mapping together with the first-order optimality condition associated with the convex minimization problem
$\min_{y\in\mathbb{R}^l}\left\{\Upsilon(y)+\frac{1}{2\mu}\|\tilde{w}-y\|^{2}\right\}$, where $\tilde{w}=w^*-\mu\nabla h(w^*)$,
we infer that \eqref{eqcop1} is satisfied if
\begin{equation}\label{eqcop2}
0 \in \nabla h(w^*) + \partial \Psi(w^*),
\end{equation}
Here, $\partial \Psi(w^*)$ represents the subdifferential of $\Psi$ at $w^*$ in the framework of convex analysis. It follows from \cite{29} that condition \eqref{eqcop2} is necessary and sufficient for $w^*$ to be an optimal solution of COP \eqref{eqCOP}. This completes the proof.
\end{proof}

\subsection{Minimax Optimization Problems}
Here, we establish the convergence of TVPNM for MOPs via MVIPs. We first show the relation between $\operatorname{MVIP}(\Upsilon,\Omega)$ in \eqref{MV} and MOPs, that is
\begin{equation}\label{eqMO1}
    \inf_{w\in \mathbb{R}^l} \sup_{y \in \mathbb{R}^m} \mathcal{M}(w,y).
\end{equation}
where $\mathcal{M} : \mathbb{R}^l \times \mathbb{R}^m \to \mathbb{R}$. Recently, the MOP has received growing attention in robust optimization, game theory, and machine learning, particularly in the context of adversarial training and economic management (see \cite{30}). A central issue in MOP is achieving fast and efficient solution methods. This motivates studying MOP~\eqref{eqMO1} beyond the convex-concave regime. We specifically consider the case $l=m$ with
\(
\mathcal{M}(w,y)=\langle \Upsilon(w),w-y\rangle+\Psi(w)-\Psi(y), \quad\forall w,y\in\mathbb{R}^l.
\) The result below demonstrates that $\operatorname{MVIP}(\Upsilon, \Omega)$ in \eqref{MV} is equivalent to MOP in \eqref{eqMO1}.\\

\begin{lemma}\label{lemMOP1}
A point $w^*$ solves $\operatorname{MVIP}(\Upsilon,\Omega)$ in \eqref{MV} iff
\[
\inf_{w \in\mathbb{R}^{l}} \sup_{y\in\mathbb{R}^{l}} \mathcal{M}(w,y)=0,
\]
with $w^*$ attaining the infimum.
\end{lemma}

\begin{proof}
Let $h(w) := \sup_{y \in \mathbb{R}^l} \mathcal{M}(w,y)$. Observe that $\mathcal{M}(w,w) = 0$ for all $w\in \mathbb{R}^l$. Then $h(w) \geq 0$ for all $w\in \mathbb{R}^l$. Consequently, one has
\begin{align*}w^* \in  \operatorname{Sol}(\Upsilon, \Omega) &\Leftrightarrow \langle \Upsilon(w^*), y - w^* \rangle + \Psi(y) - \Psi(w^*) \geq 0, \quad \forall y \in \mathbb{R}^l \\
&\Leftrightarrow \langle \Upsilon(w^*), w^* - y \rangle + \Psi(w^*) - \Psi(y) \leq 0, \quad \forall y \in \mathbb{R}^l \\
&\Leftrightarrow h(w^*) = \sup_{y \in \mathbb{R}^l} \mathcal{M}(w^*, y) \leq 0 \\
&\Leftrightarrow h(w^*) = \sup_{y \in \mathbb{R}^l} \mathcal{M}(w^*, y) = 0 \quad [\text{by } h(w^*) \geq 0] \\
&\Leftrightarrow \min_{w\in \mathbb{R}^l} \sup_{y \in \mathbb{R}^l} \mathcal{M}(w, y) = 0,
\end{align*}
as claimed. This proves the result.
\end{proof}

\begin{lemma}\label{MOPlemm}
Suppose that $\Upsilon$ is a $\zeta$-pseudomonotone and continuous function. Then
\[
w^* \in  \operatorname{Sol}(\Upsilon, \Omega) \Leftrightarrow \min_{w\in \mathbb{R}^l} \sup_{y \in \mathbb{R}^l} \mathcal{M}(w,y) = \max_{y \in \mathbb{R}^l} \inf_{w\in \mathbb{R}^l} \mathcal{M}(w,y) = 0.\]
\end{lemma}

\begin{proof}
From the $\zeta$-pseudomonotonicity and continuity of $\Upsilon$, it implies that $w^* \in \operatorname{Sol}(\Upsilon, \Omega)$ iff it is a solution of \eqref{MV}. Note that $\inf_{w\in \mathbb{R}^l} \mathcal{M}(w,y) \leq \mathcal{M}(y,y) = 0$ for all $y \in \mathbb{R}^l$. So,
\[
\langle \Upsilon(w^*), y - w^* \rangle + \Psi(y) - \Psi(w^*) \geq 0, \quad \forall y \in \mathbb{R}^l.
\]
\begin{equation}\label{eqMOP1}
    \Leftrightarrow \inf_{w\in \mathbb{R}^l} \mathcal{M}(w,w^*) = 0 = \max_{y \in \mathbb{R}^l} \inf_{w\in \mathbb{R}^l} \mathcal{M}(w,y).
\end{equation}
The result now follows from \eqref{eqMOP1} and Lemma~\ref{lemMOP1}, which completes the proof.
\end{proof}
In view of the connection between $\operatorname{MVIP}(\Upsilon, \Omega)$~\eqref{MV} and MOP~\eqref{eqMO1}, From Lemma \eqref{MOPlemm}, we have \(z^{*} := (w^*,y^{*})\in\Omega\times\Omega=\Omega\) as a solution of MOP
\[
\inf_{w \in \Omega_{1}} \ \sup_{y\in \Omega_{2}} \mathcal{M}(w,y).
\]
Moreover, assume the operators:
\[
\nabla_w \mathcal{M} : \mathbb{R}^l \to \mathbb{R}^l
\quad \text{and} \quad
-\nabla_y \mathcal{M} : \mathbb{R}^m \to \mathbb{R}^m.
\]
Then the mapping
\(
\Upsilon := (\nabla_w \mathcal{M}, -\nabla_y \mathcal{M})
\)
is $\Psi$-strongly pseudomonotone with constant $\zeta:=\min\{\zeta_1,\zeta_2\}$, where
$\Psi(w,y)=\Psi_1(w)+\Psi_2(y)$ and $\Psi_1$, $\Psi_2$ are strongly pseudomonotone
with constants $\zeta_1>0$ and $\zeta_2>0$, respectively.
\\[2mm]
For clarity, we adopt the following standing assumptions:
\\[2mm]
\noindent
 \textbf{A4.} The mapping $\mathcal{M}$ satisfies quasi-convexity in its first argument and quasi-concavity in its second argument.
\\[2mm]
\textbf{A5.} The mapping $\nabla \mathcal{M}$ is $L$-Lipschitz continuous, where $L>0$.
\\[2mm]
 \textbf{A6.} The mapping\(\Upsilon= (\nabla_w \mathcal{M},-\nabla_y \mathcal{M}):\mathbb{R}^l\times \mathbb{R}^m\to \mathbb{R}^{l+m}\) is \(\Psi\)-strongly pseudomonotone with constant \(\zeta>0\), where \(\Psi(w,y):=\Psi_1(w)+\Psi_2(y)\).\\[2mm]
The next result follows immediately from Theorem~\eqref{thmFT}.\\
\begin{corollary}
When assumptions $\textbf{A3}$--$\textbf{A6}$ hold, global predefined-time stability with respect to $\mathcal{T}_p$ is guaranteed for the solution $z^{*}=(w^*,y^{*})$ of the general MOP, and the settling time is given by
\(\frac{1}{\Theta_2(p_2-1)} < \int_{t_0}^{+\infty} \Upsilon_2(s)\,ds,\)
and
\(
\mathcal{F}_1\circ \mathcal{F}_2^{-1}
\!\left(\frac{1}{\Theta_2(p_2-1)}\right)
+\frac{1}{\Theta_1(1-p_1)}
<
\int_{t_0}^{+\infty} \Upsilon_1(s)\,ds,\)
where
\(
p_1=\frac{1+\rho_1}{2}\in(0,1),\qquad
p_2=\frac{1+\rho_2}{2}>1
\),
\(
\Theta_1 =2^{p_1}(1- \Xi),\quad
\Theta_2 =2^{p_2}(1- \Xi)^{\rho_2},
\quad
\Xi=\frac{1}{\sqrt{1+2\mu\zeta-\mu^2L^2}},
\)
and
\(
\mathcal{F}_i(t)=\int_{t_0}^t \Upsilon_i(s)\,ds,\quad i=1,2.
\)  
\end{corollary}
\section{Numerical Example}
In this section, we illustrate the effectiveness of the proposed fixed-time convergent FXT-PDM \eqref{seqODE2} for solving the MVIP \eqref{MV}.

\begin{example} Consider the nonlinear complementarity problem (NCP) with \(\Upsilon(w) := \phi(w) + Bx + p\), where
\(\phi(w) := (\arctan(w_1), \arctan(w_2), \arctan(w_3), \arctan(w_4), \arctan(w_5))\).
The matrix \(B \in \mathbb{R}^{5\times5}\) is given by
\[
B =
\begin{pmatrix}
3 & -1 & 0 & 0 & 0 \\
-1 & 3 & -1 & 0 & 0 \\
0 & -1 & 3 & -1 & 0 \\
0 & 0 & -1 & 3 & -1 \\
0 & 0 & 0 & -1 & 3
\end{pmatrix}.
\]
The equilibrium point is chosen as
\[
w^* = (1.0,\; 2.8,\; 0.5,\; 0,\; 1.5)^\top.
\]
The vector $p$ is defined as
\(
p=-\arctan(w^*)-Bw^*
=(-0.9854,\,-8.1278,\,0.8364,\,2.0000,\,-5.4828)^\top,
\)
which ensures that $w^*$ satisfies the NCP. It can be readily verified that
assumptions A1--A3 are satisfied. Next, the convergence behavior and FXT convergence properties are
investigated. For $n=5$, $\rho_1=0.5$, $\rho_2=1.6$,
$\Upsilon_1(t)=\Upsilon_2(t)=\beta_1=\beta_2=50$,
$\Upsilon_3(t)=\beta_3=20$, $q=0.3$, and $\mu=0.5$ in the FXT-PDM \eqref{seqODE2}, the following five fixed initial values are chosen:
\(
x^{(1)}(0) = (5,\;1,\;2,\;4,\;-2.5)^\top, \
x^{(2)}(0) = (-4,\;2.5,\;2,\;3,\;1)^\top, \
x^{(3)}(0) = (3.5,\;-3,\;4,\;2,\;2.5)^\top, \
x^{(4)}(0) = (2,\;4,\;3,\;-1,\;3.5)^\top, \
x^{(5)}(0) = (1.5,\;5,\;-0.5,\;0.2,\;4)^\top.\)
Figure~\ref{fig:1a} illustrates the state trajectories of FXT-PDM \eqref{seqODE2} for the above initial values, where all state components converge to the equilibrium point \(w^*\).
Figure~\ref{fig:1b} depicts the convergence responses of the errors between \(w^*\) and the solution \(w(t)\) of FXT-PDM \eqref{seqODE2}.
\begin{figure}[H]
\centering
\begin{subfigure}{0.48\textwidth}
    \centering
    \includegraphics[width=\linewidth]{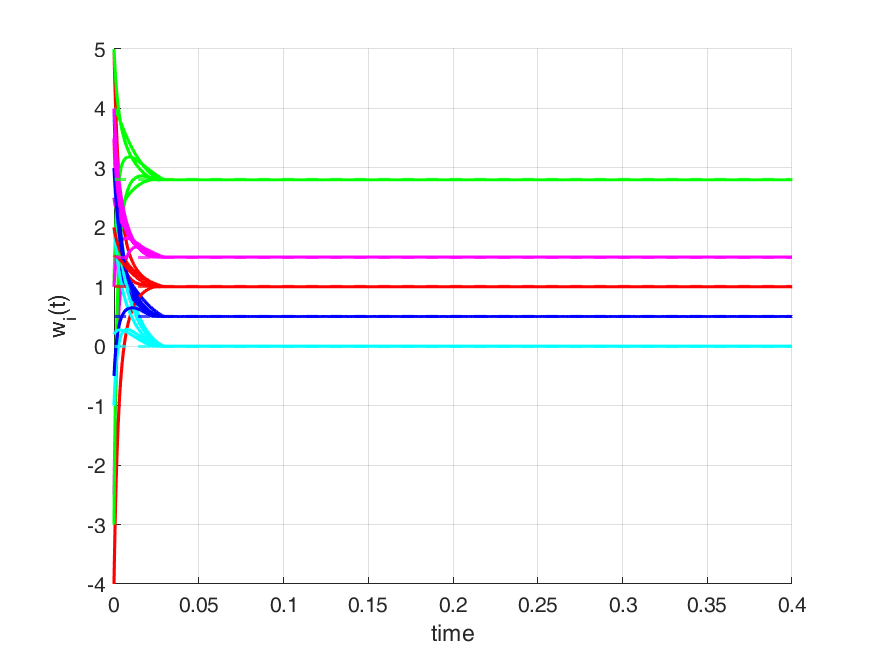}
    \caption{State trajectories}
    \label{fig:1a}
\end{subfigure}
\hfill
\begin{subfigure}{0.48\textwidth}
    \centering
    \includegraphics[width=\linewidth]{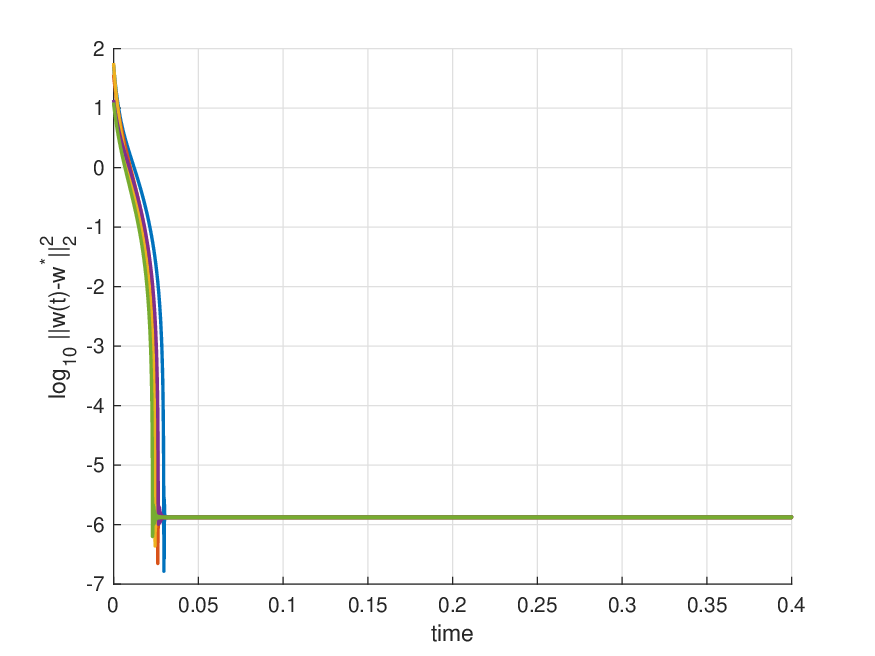}
    \caption{Error convergence}
    \label{fig:1b}
\end{subfigure}
\caption{(a) State trajectories illustrating the convergence behavior of the PNM \eqref{seqODE2} for five distinct initial conditions;
(b) evolution of the error $\log_{10}\|w(t)-w^*\|_2^2$ associated with the PNM \eqref{seqODE2} for the same set of initial conditions.}
\label{fig:example1}
\end{figure}
\begin{figure}[H]
\centering
\begin{subfigure}{0.48\textwidth}
    \centering
    \includegraphics[width=\linewidth]{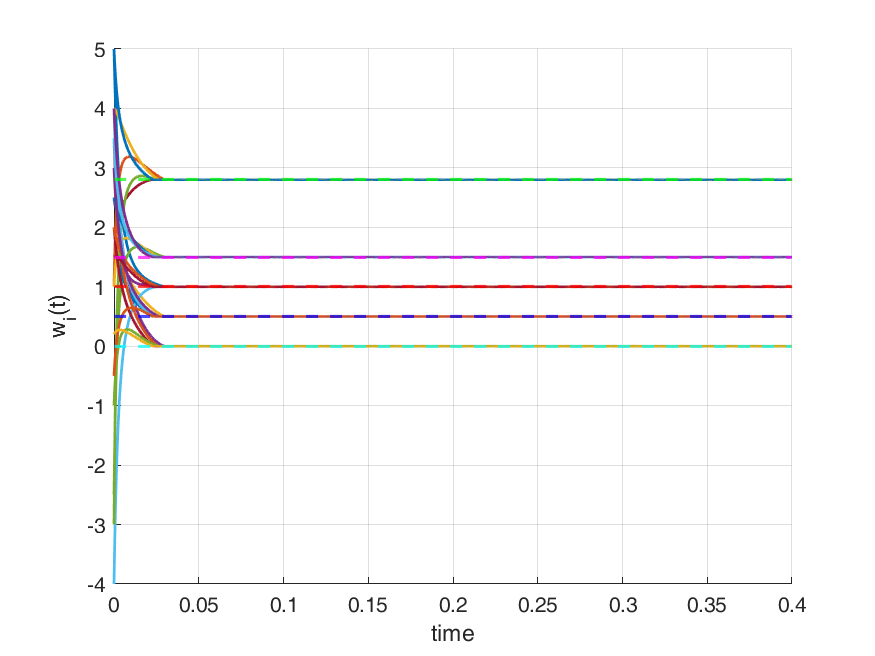}
    \caption{State trajectories}
    \label{fig:2a}
\end{subfigure}
\hfill
\begin{subfigure}{0.48\textwidth}
    \centering
    \includegraphics[width=\linewidth]{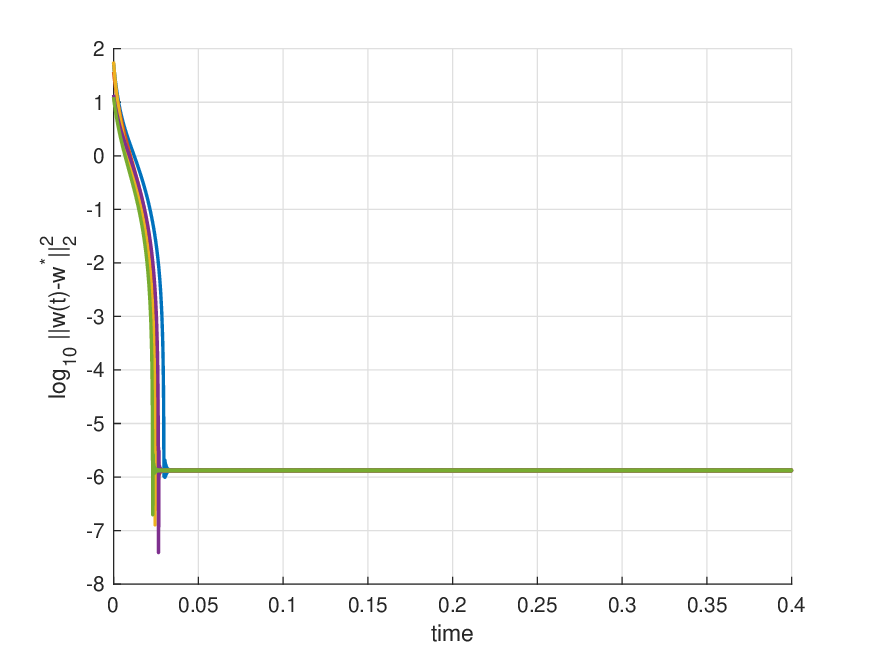}
    \caption{Error convergence}
    \label{fig:2b}
\end{subfigure}
\caption{(a) Illustrates the state trajectories of the FXPDS under a state-dependent disturbance. The disturbance affects the convergence behavior, leading to bounded trajectories that approach a neighborhood of the equilibrium point.
(b) illustrates the evolution of the error 
$\log_{10}\|w(t)-w^*\|_2^2$, which demonstrates the practical stability of the proposed system under disturbance.}
\label{fig:exampl2}
\end{figure}
\end{example}

\section{Conclusion}\label{p3sec5}

This paper introduces a class of original fixed-time convergent proximal neurodynamic models strategies for solving \(\operatorname{MVIP}(\Upsilon, \Omega )\) \eqref{MV}. A Lyapunov-based approach is employed to rigorously verify the fixed-time convergence of the proposed neurodynamic models toward the solution of MVI. Under the assumptions of strong pseudo-monotonicity and Lipschitz-type continuity. The model's robustness to bounded disturbances was also analyzed, ensuring reliable performance in practical settings. We also apply the proposed method to COPs and MOPs. Numerical experiments confirmed the fast convergence and effectiveness of the proposed approach, demonstrating its applicability to MVIP. 

Future research may focus on extending the proposed framework to stochastic variational inequalities and different types of neurodynamic systems, distributed and networked systems, and applications in signal processing, game theory, and machine learning.

\section*{CRediT authorship contribution statement}
V. K. Khan: Writing-original draft, writing-review and editing, methodology, software, conceptualization; V. K. Varshney: visualizations, investigation, formal analysis; M. K. Ahmad: supervision, investigation, formal analysis, writing review and editing. All authors have read and approved the final version of the manuscript for publication.
\section*{Use of AI tools declaration}
The authors declare that they have not used AI tools in the creation of this article.
\section*{Declaration of interests} 
The authors declare that they have no known competing financial interests or personal relationships 
that could have appeared to influence the work reported in this paper.
\section*{Conflict of interest}
The authors declare no conflicts of interest in this paper.
\section*{Funding}
This research received no external funding.

\end{document}